\documentclass[12pt,fleqn,leqno]{amsart}
\usepackage{amsfonts,amssymb}
\usepackage{mathrsfs}
\usepackage[bb=boondox]{mathalfa} 

\usepackage[svgnames]{xcolor}
\usepackage[colorlinks,linkcolor=blue,citecolor=Green]{hyperref}
\usepackage{titlesec}
\usepackage{tikz-cd}
\usetikzlibrary{decorations.pathmorphing}

\titleformat{\section}[block]
 {\bfseries}
 {\thesection.}
 {\fontdimen2\font}
 {}

\newtheorem{theorem}{Theorem}[section]
\newtheorem{corollary}[theorem]{Corollary}
\newtheorem{proposition}[theorem]{Proposition}

\theoremstyle{definition}
\newtheorem{remark}[theorem]{Remark}

\DeclareMathOperator{\N}{\mathbb{N}}
\DeclareMathOperator{\R}{\mathbb{R}}
\DeclareMathOperator{\B}{\mathbb{B}}
\DeclareMathOperator{\s}{\mathbb{S}}
\DeclareMathOperator{\uhr}{\upharpoonright} 
\DeclareMathOperator{\st}{st} 
\DeclareMathOperator\smap{\leadsto}
\DeclareMathOperator\card{Card}
\DeclareMathOperator\conv{conv}
\DeclareMathOperator\coz{coz}
\newcommand{\embed}[1]{\stackrel{#1}{\hookrightarrow}}
\DeclareMathOperator\car{car}
\DeclareMathOperator\supp{supp}

\renewcommand{\emptyset}{\varnothing}
\numberwithin{equation}{section}

\begin{document}

\author{Valentin Gutev}

\address{Department of Mathematics, Faculty of Science, University of
   Malta, Msida MSD 2080, Malta}
\email{valentin.gutev@um.edu.mt}



\title[Constructing Selections Stepwise Over Skeletons of Nerves]{Constructing
  Selections Stepwise Over Skeletons of Nerves of Covers}

\begin{abstract}
  It is given a simplified and self-contained proof of the classical
  Michael's finite-dimensional selection theorem. The proof is based on
  approximate selections constructed stepwise over skeletons of nerves
  of covers. The method is also applied to simplify the proof of the
  Schepin--Brodsky's generalisation of this theorem.
\end{abstract}

\date{\today}
\maketitle

\section{Introduction}

All spaces in this paper are Hausdorff topological spaces. We will use
$\Phi:X\smap Y$ to designate that $\Phi$ is a map from $X$ to the
nonempty subsets of $Y$, i.e.\ a \emph{set-valued mapping}. Such a
mapping is \emph{lower semi-continuous}, or l.s.c., if the set \[
\Phi^{-1}[U]=\{x\in X:\Phi(x)\cap U\neq \emptyset\}
\]
is open in $X$, for every open $U\subset Y$. Also, let us recall that a
map $f:X\to Y$ is a \emph{selection} for $\Phi:X\smap Y$ if $f(x)\in
\Phi(x)$, for all $x\in X$.\medskip

Let $n\geq-1$. A family $\mathscr{S}$ of subsets of a space $Y$ is
\emph{equi-$LC^n$} \cite{michael:56b} if every neighbourhood
$U$ of a point $y\in \bigcup\mathscr{S}$ contains a neighbourhood $V$
of $y$ such that for every $S\in \mathscr{S}$, every continuous map
$g:\s^k\to V\cap S$ of the $k$-sphere $\s^k$, $k\leq n$, can be
extended to a continuous map $h:\B^{k+1}\to U\cap S$ of the
$(k+1)$-ball $\B^{k+1}$.  A space $S$ is called $C^n$ if for every
$k\leq n$, every continuous map $g:\s^k\to S$ can be extended to a
continuous map $h:\B^{k+1}\to S$. In these terms, a family
$\mathscr{S}$ of subsets of $Y$ is equi-$LC^{-1}$ if it consists of
nonempty subsets; similarly, each nonempty subset $S\subset Y$ is
$C^{-1}$.\medskip

Let $\mathscr{F}(Y)$ be the collection of all nonempty closed subsets
of a space $Y$. The following theorem was proved by Ernest Michael,
see \cite[Theorem 1.2]{michael:56b}, and is commonly called the
\emph{finite-dimensional selection theorem}.

\begin{theorem}
  \label{theorem-st-app-v10:1}
  Let $X$ be a paracompact space with $\dim(X)\leq n+1$, $Y$ be a
  com\-pletely metrizable space, and $\mathscr{S}\subset \mathscr{F}(Y)$
  be an equi-$LC^n$ family such that each $S\in \mathscr{S}$ is
  $C^n$. Then each l.s.c.\ mapping $\Phi:X\to \mathscr{S}$ has a
  continuous selection.
\end{theorem}

The original proof of Theorem \ref{theorem-st-app-v10:1} in
\cite{michael:56b} takes up most of that paper, and is accomplished in
6 steps. Other proofs of this theorem can be found in the monograph
\cite{repovs-semenov:98}, and the book
\cite{zbMATH00193669}. Actually, in \cite{repovs-semenov:98} are given
two different approaches to obtain the theorem --- the one which
follows the original Michael's proof, and another one based on
filtrations \cite{schepin-brodsky:96}. Other proofs were given by
other authors, see e.g. \cite{MR2192951} and \cite{gutev:05}.
However, what all these proofs have in common is that they may somehow
discourage the casual reader and make Theorem
\ref{theorem-st-app-v10:1} not so accessible to wider audience. The
main purpose of this paper is to fill in this gap, and present a
simplified and self-contained proof of this theorem. \medskip

The paper is organised as follows. The next section contains a brief
review of canonical maps and partitions of unity, which is essential
for the proper understanding of any of the available proofs of Theorem
\ref{theorem-st-app-v10:1}. In this regard, let us explicitly remark
that these considerations were not made readily available in previous
proofs, so they are now included to make the exposition
self-contained. The essential preparation for the proof of Theorem
\ref{theorem-st-app-v10:1} starts in Section
\ref{sec:asph-sequ-select}, which contains a selection theorem for
finite aspherical sequences of lower locally constant mappings
(Theorem \ref{theorem-nerves-v2:5}). This theorem is similar to a
theorem of Uspenskij, see \cite[Theorem 1.3]{uspenskij:98}, and
represents a relaxed version of another theorem proved by the author,
see \cite[Theorem 3.1]{gutev:05}. Section \ref{sec:gener-asph-sequ}
contains several simple constructions of finite aspherical sequences
of sets providing the main interface between such sequences of sets
and the property of equi-$LC^n$. Finally, the proof of Theorem
\ref{theorem-st-app-v10:1} is accomplished in Section
\ref{sec:select-lcn-valu}. It is based on two constructions which are
also present in Michael's proof. The one, Proposition
\ref{proposition-st-app-v15:2}, relates l.s.c.\ mappings to lower
locally constant mappings; the other --- Proposition
\ref{proposition-st-app-v15:3}, relates selections for lower locally
constant mappings to approximate selections for l.s.c.\
mappings. These constructions are applied together with Theorem
\ref{theorem-nerves-v2:5} to deal with two selection properties of
l.s.c.\ equi-$LC^n$-valued mappings, see Theorems
\ref{theorem-st-app-v9:1} and \ref{theorem-st-app-v17:1}. The
proof of Theorem \ref{theorem-st-app-v10:1} is then obtained as an
immediate consequence of these properties.

\section{Canonical maps and partitions of unity}
\label{sec:canon-maps-part}

The \emph{cozero set}, or the \emph{set-theoretic support}, of a
function $\xi:X\to \R$ is the set $\coz(\xi)=\{x\in X:\xi(x)\neq 0\}$.
A collection $\xi_a:X\to [0,1]$, $a\in \mathscr{A}$, of continuous
functions on a space $X$ is a \emph{partition of unity} if
$\sum_{a\in \mathscr{A}}\xi_a(x)=1$, for each $x\in X$.  Here,
``$\sum_{a\in \mathscr{A}}\xi_a(x)=1$'' means that only countably many
functions $\xi_a$'s do not vanish at $x$, and the series composed by
them is convergent to 1.  For a cover $\mathscr{U}$ of a space $X$, a
partition of unity $\{\xi_U:U\in \mathscr{U}\}$ on $X$ is
\emph{index-subordinated} to $\mathscr{U}$ if $\coz(\xi_U)\subset U$,
for each $U\in \mathscr{U}$, see Remark
\ref{remark-st-app-vgg-rev:1}. The following theorem is well known, it
is a consequence of Urysohn's characterisation of normality
\cite{MR1512258} and the Lefschetz lemma \cite{MR0007093}.

\begin{theorem}
  \label{theorem-st-app-v3:1}
  Every locally finite open cover of a normal space has an
  index-subordinated partition of unity.
\end{theorem}

A partition of unity $\{\xi_a:a\in \mathscr{A}\}$ on a space $X$ is
called \emph{locally finite} if $\{\coz(\xi_a):a\in \mathscr{A}\}$ is
a locally finite cover of $X$. Complementary to Theorem
\ref{theorem-st-app-v3:1} is the following important property of
partitions of unity; it follows from a construction of M. Mather, see
\cite[Lemma]{MR0281155} and \cite[Lemma 5.1.8]{engelking:89}.

\begin{theorem}
  \label{theorem-st-app-vgg-rev:1}
  If a cover $\mathscr{U}$ of a space $X$ has an index-subordinated
  partition of unity, then $\mathscr{U}$ also has an
  index-subordinated locally finite partition of unity. 
\end{theorem}

By a \emph{simplicial complex} we mean a collection $\Sigma$ of
nonempty finite subsets of a set $S$ such that $\tau\in \Sigma$,
whenever $\emptyset\neq \tau\subset \sigma\in \Sigma$. The set
$\bigcup\Sigma $ is the \emph{vertex set} of $\Sigma$, while each
element of $\Sigma$ is called a \emph{simplex}. The
\emph{$k$-skeleton} $\Sigma^k$ of $\Sigma$ ($k\geq 0$) is the
simplicial complex
$\Sigma^{k}=\{\sigma\in \Sigma:\card(\sigma)\leq k+1\}$, where
$\card(\sigma)$ is the cardinality of $\sigma$. In the sequel, for
simplicity, we will identify the vertex set of $\Sigma$ with its
$0$-skeleton $\Sigma^0$.  In these terms, a \emph{simplicial map}
$g:\Sigma_1\to \Sigma_2$ is a map $g:\Sigma_1^0\to \Sigma_2^0$ between
the vertices of simplicial complexes $\Sigma_1$ and $\Sigma_2$ such
that $g(\sigma)\in \Sigma_2$, for each $\sigma\in \Sigma_1$. If
$g:\Sigma_1\to \Sigma_2$ is a simplicial map and
$g:\Sigma_1^0\to \Sigma_2^0$ is bijective, then the inverse $g^{-1}$
is also a simplicial map, and we say that $g$ is a \emph{simplicial
  isomorphism}.\medskip

The set $\Sigma_S$ of all nonempty finite subsets of a set $S$ is a
simplicial complex. Another natural example is the \emph{nerve}
$\mathscr{N}(\mathscr{U})$ of a cover $\mathscr{U}$ of a set $X$,
which is the subcomplex of $\Sigma_\mathscr{U}$ defined by
\begin{equation}
  \label{eq:st-app-vgg-rev:8}
\mathscr{N}(\mathscr{U})= \left\{\sigma\in
  \Sigma_\mathscr{U}:\bigcap\sigma\neq\emptyset\right\}.  
\end{equation}
The $k$-skeleton of $\mathscr{N}(\mathscr{U})$ is denoted by
$\mathscr{N}^k(\mathscr{U})$, and the vertex set
$\mathscr{N}^0(\mathscr{U})$ of $\mathscr{N}(\mathscr{U})$ is actually
$\mathscr{U}$ because we can always assume that
$\emptyset\notin \mathscr{U}$.\medskip

For a set $\mathscr{A}$, let $\ell_1 (\mathscr{A})$ be the linear
space of all functions $y:\mathscr{A}\to \R$ with
${\sum _{a \in \mathscr{A}} |y (a)| < \infty}$. In fact,
$\ell_1(\mathscr{A})$ is a Banach space when equipped with the norm
$\|y\|_1 = \sum _{a\in \mathscr{A}} |y (a)| $, but this will play no
role in the paper.  The vertex set $\Sigma^0$ of a simplicial complex
$\Sigma$ is a linearly independent subset of $\ell_1(\Sigma^0)$, where
each $v\in \Sigma^0$ is identified with its characteristic function
$v:\Sigma^0\to \{0,1\}$, namely with the function $v(u)=0$ for
$u\neq v$, and $v(v)=1$.  Then to each $\sigma\in \Sigma$ one can
associate the \emph{geometric simplex} $|\sigma|=\conv(\sigma)$, which
is the convex hull of $\sigma$. Thus, $|\sigma|$ is a
\emph{$k$-dimensional simplex} if and only if
$\hbox{Card}(\sigma)=k+1$. The set
$|\Sigma|=\bigcup_{\sigma\in \Sigma}|\sigma|\subset\ell_1(\Sigma^0)$
is called the \emph{geometric realisation} of $\Sigma$. As a
topological space, we will consider $|\Sigma|$ endowed with the
\emph{Whitehead topology} \cite{MR1576810,MR0030759}. In this
topology, a subset $U\subset |\Sigma|$ is open if and only if
$U\cap |\sigma|$ is open in $|\sigma|$, for every $\sigma\in \Sigma$.
Let us explicitly remark that the Whitehead topology on $|\Sigma|$ is
not necessarily the subspace topology on $|\Sigma|$ as a subset of the
Banach space $\ell_1\left(\Sigma^0\right)$. However, both topologies
coincide on each geometric simplex $|\sigma|$, for $\sigma\in
\Sigma$. \medskip

If $p\in |\sigma|$ for some $\sigma\in \Sigma$, then $p$ is both an
element $p\in \ell_1(\Sigma^0)$ and a unique convex
combination of the elements of
$\sigma\subset \Sigma^0\subset
\ell_1(\Sigma^0)$. Hence,
the geometric realisation $|\Sigma|$ is the set of all
$p\in\ell_1(\Sigma^0)$ such that
\begin{equation}
  \label{eq:st-app-vgg-rev:1}
  p(v)\geq 0,\ v\in \Sigma^0,\quad\text{and}\quad
  \coz(p)=\left\{v\in \Sigma^0: p(v)>
    0\right\}\in \Sigma.
\end{equation}
Here, $p(v)$ is called the $v$-th \emph{barycentric} (or
\emph{affine}) \emph{coordinate of} ${p\in |\Sigma|}$, while the
simplex $\coz(p)\in \Sigma$ is called the \emph{carrier} of $p$, and
denoted by $\car(p)=\coz(p)$. Since the representation
$p= \sum_{v\in \car(p)}p(v)\cdot v$ is unique, the carrier $\car(p)$
is the minimal simplex of $\Sigma$ with the property that
$p\in|\car(p)|$.\medskip

To each vertex ${v}\in \Sigma^0$, we can now associate the function
$\alpha_v:|\Sigma|\to [0,1]$, defined by
\begin{equation}
  \label{eq:st-app-vgg-rev:3}
  \alpha_v(p)=p(v),\quad \text{for every $p\in
    |\Sigma|$.} 
\end{equation}
  It is
called the $v$-th \emph{barycentric coordinate function} and
is continuous being affine on each simplex $|\sigma|$, for
$\sigma\in \Sigma$. The cozero
set $\coz(\alpha_v)$ of $\alpha_v$ is called the
\emph{open star} of the vertex $v\in \Sigma^0$, and
denoted by
\begin{equation}
  \label{eq:st-app-v4:3}
  \st\langle v\rangle
  =\big\{p\in |\Sigma|: \alpha_v(p)>0\big\}.
\end{equation}
Clearly, the open star $\st\langle v\rangle$ is open in
$|\Sigma|$ because $\alpha_v$ is continuous.  The following
proposition is an immediate consequence of \eqref{eq:st-app-vgg-rev:1},
\eqref{eq:st-app-vgg-rev:3} and \eqref{eq:st-app-v4:3}.

\begin{proposition}
  \label{proposition-st-app-v2:1}
  If $\Sigma$ is a simplicial complex, then the collection
  $\left\{\alpha_v:v\in \Sigma^0\right\}$ is a
  partition of unity on $|\Sigma|$ with
  $\coz(\alpha_v)=\st\langle v\rangle$, for each
  $v\in \Sigma^0$.
\end{proposition}

We now turn to the other essential concept in this section. For a
cover $\mathscr{U}$ of a space $X$, a continuous map
$f:X\to |\mathscr{N}(\mathscr{U})|$ is called \emph{canonical for
  $\mathscr{U}$} if
\begin{equation}
  \label{eq:st-app-v1:3}
  f^{-1}(\st\langle U\rangle)\subset U,\quad \text{for every $U\in
    \mathscr{U}$.} 
\end{equation}
Canonical maps are essentially partitions of unity, which are
index-subordinated to the corresponding cover of the space.

\begin{theorem}
  \label{theorem-st-app-v12:1}
    A cover $\mathscr{U}$ of a space $X$ has an index-subordinated
  partition of unity if and only if $\mathscr{U}$ has a canonical
  map. 
\end{theorem}

\begin{proof}
  Let $\mathscr{U}$ be a cover of $X$ and $\alpha_U$, $U\in \mathscr{U}$,
  be the barycentric coordinate functions of
  $|\mathscr{N}(\mathscr{U})|$.\smallskip

  Suppose that $f:X\to |\mathscr{N}(\mathscr{U})|$ is a canonical map
  for $\mathscr{U}$. Since $f$ is continuous, by Proposition
  \ref{proposition-st-app-v2:1},
  $\{\alpha_U\circ f:U\in \mathscr{U}\}$ is a partition of unity on
  $X$. By the same proposition and \eqref{eq:st-app-v1:3}, we also
  have that
  \[
  \coz(\alpha_U\circ f)=f^{-1}(\coz(\alpha_U))=f^{-1}(\st\langle
  U\rangle)\subset U,\quad U\in \mathscr{U}.
  \]
  
  Conversely, suppose that $\mathscr{U}$ has an index-subordinated
  partition of unity. Then by Theorem \ref{theorem-st-app-vgg-rev:1},
  $\mathscr{U}$ also has an index-subordinated locally finite
  partition of unity $\{\xi_U:U\in \mathscr{U}\}$. For each $x\in X$,
  let $\sigma_\xi(x)\in \mathscr{N}(\mathscr{U})$ be the simplex
  determined by the point $x$ and the functions $\xi_U$,
  $U\in \mathscr{U}$, namely
  $\sigma_\xi(x)=\{U\in \mathscr{U}:\xi_U(x)>0\}$.  Next, define a map
  $f:X\to |\mathscr{N}(\mathscr{U})|$ by
  \begin{equation}
    \label{eq:st-app-v3:1}
  f(x)=\sum_{U\in
    \sigma_\xi(x)}\xi_U(x)\cdot U,\quad x\in X.
\end{equation}

Since $\{\xi_U:U\in \mathscr{U}\}$ is a locally finite partition of
unity, each point $p\in X$ has a neighbourhood $V_p\subset X$ such
that
$\mathscr{U}_p=\{U\in \mathscr{U}: V_p\cap \coz(\xi_U)\neq
\emptyset\}$ is a finite set. According to \eqref{eq:st-app-v3:1},
this implies that
$f(V_p)\subset |\mathscr{N}(\mathscr{U}_p)|\subset
\ell_1(\mathscr{U}_p)$. However, $\ell_1(\mathscr{U}_p)$ is now the
usual Euclidean space $\R^{\mathscr{U}_p}$ because $\mathscr{U}_p$ is
a finite set.  For the same reason, $\mathscr{N}(\mathscr{U}_p)$ has
finitely many simplices. Therefore, the Whitehead topology on
$|\mathscr{N}(\mathscr{U}_p)|$ is the subspace topology on
$|\mathscr{N}(\mathscr{U}_p)|$ as a subset of $\R^{\mathscr{U}_p}$.
 Since each
function $\xi_U=\alpha_U\circ f$, $U\in \mathscr{U}_p$, is continuous,
so is the restriction $f\uhr V_p$. This shows that $f$ is continuous
as well. Finally, let $U\in \mathscr{U}$ and
$x\in f^{-1}\left(\st\langle U\rangle\right)$. Then
$f(x)\in \st\langle U\rangle$ and by \eqref{eq:st-app-v4:3} and
\eqref{eq:st-app-v3:1}, we get that
$\xi_U(x)=\alpha_U(f(x))>0$. Accordingly, $U\in \sigma_\xi(x)$ which
implies that $x\in U$ because $\coz(\xi_U)\subset U$. Thus, $f$ is
canonical for $\mathscr{U}$, see \eqref{eq:st-app-v1:3}.
\end{proof}

Canonical maps will be involved in the proof of Theorem
\ref{theorem-st-app-v10:1} with two properties, which are briefly
discussed below.\medskip

For a simplicial complex $\Sigma$, as mentioned before, the carrier
$\car(p)$ of a point $p\in |\Sigma|$ is the minimal simplex of
$\Sigma$ with $p\in |\car(p)|$, see \eqref{eq:st-app-vgg-rev:1}.
According to \eqref{eq:st-app-vgg-rev:3} and \eqref{eq:st-app-v4:3},
it has the following natural representation
\begin{equation}
  \label{eq:st-app-v4:2}
  \car(p)=\left\{v\in \Sigma^0:
    p\in \st\langle v\rangle\right\}.
\end{equation}

For a cover $\mathscr{U}$ of $X$ and $x\in X$, we will associate the
simplicial complex
\begin{equation}
  \label{eq:st-app-vgg-rev:7}
  \Sigma_\mathscr{U}(x)=\left\{\sigma\in \Sigma_\mathscr{U}: x\in\bigcap
  \sigma\right\}. 
\end{equation}
According to \eqref{eq:st-app-vgg-rev:8}, we have
that $\Sigma_\mathscr{U}(x)\subset \mathscr{N}(\mathscr{U})$, for
every $x\in X$. Thus, \eqref{eq:st-app-vgg-rev:7} defines a
natural set-valued mapping $\Sigma_\mathscr{U}:X\smap
\mathscr{N}(\mathscr{U})$.  To this mapping, we will associate the
mapping $|\Sigma_\mathscr{U}|:X\smap |\mathscr{N}(\mathscr{U})|$ which
assigns to each $x\in X$ the geometric realisation
$|\Sigma_\mathscr{U}|(x)= |\Sigma_\mathscr{U}(x)|$. In terms of this
mapping, we have the following selection interpretation of canonical
maps which extends an observation of Dowker \cite{dowker:47}, see
Remark \ref{remark-st-app-vgg-rev:6}.

\begin{proposition}
  \label{proposition-st-app-v11:1}
  Let $\mathscr{U}$ be a cover of a space $X$. Then a continuous map
  $f:X\to |\mathscr{N}(\mathscr{U})|$ is canonical for $\mathscr{U}$
  if and only if $f$ is a selection for the mapping
  $|\Sigma_\mathscr{U}|:X\smap |\mathscr{N}(\mathscr{U})|$.
\end{proposition}

\begin{proof}
  Let $f$ be a canonical map for $\mathscr{U}$, and $x\in X$. Whenever
  $U\in \car(f(x))$, it follows from \eqref{eq:st-app-v4:2} that
  $f(x)\in \st\langle U\rangle$ and therefore, by
  (\ref{eq:st-app-v1:3}), $x\in U$. Thus, by
  \eqref{eq:st-app-vgg-rev:7}, $\car(f(x))\in \Sigma_\mathscr{U}(x)$
  and we have that
  $f(x)\in|\car(f(x))|\subset |\Sigma_\mathscr{U}|(x)$. Conversely,
  suppose that $f$ is as selection for $|\Sigma_\mathscr{U}|$, and
  $x\in f^{-1}\left( \st\langle U\rangle\right)$ for some
  $U\in \mathscr{U}$. Then by (\ref{eq:st-app-v4:2}), $U\in\car(f(x))$
  because $f(x)\in \st\langle U\rangle$. Moreover, $f(x)\in |\sigma|$
  for some $\sigma\in \Sigma_\mathscr{U}(x)$ because
  $f(x)\in |\Sigma_\mathscr{U}|(x)$. Since $\car(f(x))$ is the minimal
  simplex with this property, we get that
  $U\in \car(f(x))\subset \sigma$ and, therefore, $x\in U$.  That is,
  $f^{-1}(\st\langle U\rangle)\subset U$.
\end{proof}

Each {simplicial map} $g:\Sigma_1\to \Sigma_2$, between simplicial
complexes $\Sigma_1$ and $\Sigma_2$, can be extended to a continuous
map $|g|:|\Sigma_1|\to |\Sigma_2|$ which is affine on each geometric
simplex $|\sigma|$, for $\sigma\in \Sigma_1$. This map is simply
defined by
\[
  |g|(p)=\sum_{v\in \car(p)}
  \alpha_v(p)\cdot g(v),\quad p\in |\Sigma_1|.
\]

If a cover $\mathscr{V}$ of $X$ refines another cover $\mathscr{U}$,
then there exists a natural simplicial map
$r:\mathscr{N}(\mathscr{V})\to \mathscr{N}(\mathscr{U})$ with 
$V\subset r(V)$, for each $V\in \mathscr{V}$.  Such a map is commonly
called a \emph{canonical projection}, or a \emph{refining simplicial
  map}, or simply a \emph{refining map}. Canonical maps are preserved
by refinements in the following sense.

\begin{corollary}
  \label{corollary-st-app-vgg-rev:1}
  Let $\mathscr{U}$ and $\mathscr{V}$ be covers of a space $X$ such
  that $\mathscr{V}$ refines $\mathscr{U}$. If
  $r:\mathscr{N}(\mathscr{V})\to \mathscr{N}(\mathscr{U})$ is a
  refining map and $g:X\to |\mathscr{N}(\mathscr{V})|$ is canonical
  for $\mathscr{V}$, then the composite map
  $|r|\circ g:X\to |\mathscr{N}(\mathscr{U})|$ is canonical for
  $\mathscr{U}$.
\end{corollary}

\begin{proof}
  This follows from Proposition \ref{proposition-st-app-v11:1} and the
  fact that $r\left(\Sigma_\mathscr{V}(x)\right)\subset
  \Sigma_\mathscr{U}(x)$, $x\in X$, because $V\subset r(V)$ for every
  $V\in \mathscr{V}$, see \eqref{eq:st-app-vgg-rev:7}. 
\end{proof}
 
We conclude this section with several remarks. 

\begin{remark}
  \label{remark-st-app-vgg-rev:1}
  For a space $X$, the \emph{support} of a function $\xi:X\to \R$,
  called also the \emph{topological support}, is the set
  $\supp(\xi)=\overline{\coz(\xi)}$. In several sources, a partition
  of unity $\{\xi_U:U\in \mathscr{U}\}$ on a space $X$ is called
  \emph{index-subordinated} to a cover $\mathscr{U}$ of $X$ if
  $\supp(\xi_U)\subset U$, for every $U\in \mathscr{U}$; and
  $\{\xi_U:U\in \mathscr{U}\}$ is called \emph{weakly
    index-subordinated} to $\mathscr{U}$ if $\coz(\xi_U)\subset U$,
  for every $U\in \mathscr{U}$, see e.g.\ \cite{MR3099433}. However,
  these variations in the terminology do not affect the results of
  this section. Namely, if $\{\eta_U:U\in \mathscr{U}\}$ is a
  partition of unity on $X$, then $X$ also has a (locally finite)
  partition of unity $\{\xi_U:U\in \mathscr{U}\}$ with
  $\supp(\xi_U)\subset \coz(\eta_U)$, for all $U\in \mathscr{U}$,
  \cite[Proposition 2.7.4]{MR3099433}. This property is essentially
  the construction of M. Mather for proving Theorem
  \ref{theorem-st-app-vgg-rev:1}.
\end{remark}

\begin{remark}
  \label{remark-st-app-vgg-rev:2}
  Canonical maps provide an isomorphism between simplicial complexes
  and nerves of covers. Namely, if
  $\mathscr{O}_\Sigma=\big\{\st\langle {v}\rangle:
  {v}\in\Sigma^0\big\}$ is the cover of $|\Sigma|$ by the open stars
  of the vertices of a simplicial complex $\Sigma$ and
  $\sigma\subset \Sigma^0$, then $\sigma\in \Sigma$ if and
  only if
  $\bigcap_{{v}\in \sigma}\st\langle {v}\rangle\neq
  \emptyset$. That is, $\sigma\in \Sigma$ precisely when
  $\st\langle\sigma\rangle=\{\st\langle {v}\rangle:
  {v}\in \sigma\}\in \mathscr{N}(\mathscr{O}_\Sigma)$. Hence,
  $\st\langle \cdot\rangle: \Sigma\to \mathscr{N}(\mathscr{O}_\Sigma)$
  is a simplicial isomorphism and the associated map
  $|\st\langle\cdot\rangle|: |\Sigma|\to
  |\mathscr{N}(\mathscr{O}_\Sigma)|$ is both a homeomorphism and a
  canonical map for $\mathscr{O}_\Sigma$.
\end{remark}

\begin{remark}
  \label{remark-st-app-vgg-rev:6}
  In the case of a point-finite cover $\mathscr{U}$ of $X$, Proposition
  \ref{proposition-st-app-v11:1} is reduced to the following selection
  interpretation of canonical maps given by Dowker
  \cite{dowker:47}. Whenever $x\in X$, let
  $\sigma(x)=\{U\in \mathscr{U}: x\in U\}\in \mathscr{N}(\mathscr{U})$
  be the simplex determined by $x$. Then a continuous map
  $f:X\to |\mathscr{N}(\mathscr{U})|$ is canonical for $\mathscr{U}$
  if and only if $f(x)\in |\sigma(x)|$, for every $x\in X$. While
  $\sigma(x)$ is only an element of $\Sigma_\mathscr{U}(x)$, we have
  that $|\sigma(x)|=|\Sigma_\mathscr{U}|(x)$ because
  $\sigma\subset \sigma(x)$, for each
  $\sigma\in \Sigma_\mathscr{U}(x)$.
\end{remark}

\section{Aspherical sequences of mappings and selections}
\label{sec:asph-sequ-select}

A  mapping $\varphi:X\smap Y$ is  \emph{lower locally
  constant} \cite{gutev:05} if the set $\{x\in X:K\subset
\varphi(x)\}$ is open in $X$, for every compact subset $K\subset Y$.
This property appeared in a paper of Uspenskij \cite{uspenskij:98};
later on, it was used by some authors (see, for instance,
\cite{chigogidze-valov:00a,valov:00}) under the name ``strongly
l.s.c.'', while in papers of other authors strongly l.s.c.\ was
already used for a different property of set-valued mappings (see, for
instance, \cite{gutev:95e}). Every lower locally constant
mapping is l.s.c.\ but the converse fails in general and
counterexamples abound. In fact, if we consider a single-valued map
$f:X\to Y$ as a set-valued one, then $f$ is l.s.c.\ if and only if it
is continuous, while $f$ will be lower locally constant if and only if
it is locally constant. Thus, our terminology provides some natural
analogy with the single-valued case.\medskip

Let $k\ge 0$. For subsets $S, B\subset Y$, we will write that
$S\embed{k} B$ if every continuous map of the $k$-sphere in $S$ can be
extended to a continuous map of the $(k+1)$-ball in $B$. Evidently,
the relation $S\embed{k} B$ implies that $S\subset B$. Similarly,
for mappings $\varphi, \psi :X\smap Y$, we will write
$\varphi \embed{k}\psi$ to express that $\varphi(x)\embed{k}\psi(x)$,
for every $x\in X$.  In these terms, we shall say that a sequence of
mappings $\varphi_k: X\smap Y$, $0\leq k\leq n$, is
\emph{aspherical} if $\varphi_k\embed{k}\varphi_{k+1}$, for every
$k<n$.  The
following theorem will be proved in this section.

\begin{theorem}
\label{theorem-nerves-v2:5}
Let $X$ be a paracompact space with $\dim(X)\le n$, $Y$ be a space,
and $\varphi_{k}:X\smap Y$, $0\leq k\leq n$, be an aspherical sequence
of lower locally constant mappings. Then $\varphi_{n}$ has a
continuous selection.
\end{theorem}

The proof of Theorem \ref{theorem-nerves-v2:5} is based on special
skeletal selections motivated by the characterisation of canonical
maps in Proposition \ref{proposition-st-app-v11:1}.  Namely, we shall say that a mapping
$\varphi:X\smap Y$ has a \emph{$k$-skeletal selection}, $k\geq 0$, if
there exists and open cover $\mathscr{U}$ of $X$ and a continuous map
$u:|\mathscr{N}^k(\mathscr{U})|\to Y$ such that
\begin{equation}
  \label{eq:st-app-vgg-rev:9}
  u\left(|\Sigma_\mathscr{U}^k(x)|\right)\subset \varphi(x),\quad
  \text{for every $x\in X$.}
\end{equation}
Here, $\Sigma_\mathscr{U}^k(x)$ is the $k$-skeleton of the
simplicial complex $\Sigma_\mathscr{U}(x)$, see
\eqref{eq:st-app-vgg-rev:7}. In fact, just like before, one can
consider $\Sigma_\mathscr{U}^k:X\smap \mathscr{N}^k(\mathscr{U})$ as a
set-valued mapping; similarly for
$|\Sigma_\mathscr{U}^k|:X\smap |\mathscr{N}^k(\mathscr{U})|$. Then a
continuous map $u:|\mathscr{N}^k(\mathscr{U})|\to Y$ is a $k$-skeletal
selection for $\varphi$ if and only if the composite mapping
$u\circ |\Sigma_\mathscr{U}^k|:X\smap Y$ is a set-valued selection for
$\varphi:X\smap Y$, see Remark \ref{remark-st-app-vgg-rev:5}. \medskip

We proceed with the following constructions of $k$-skeletal selections
which furnish the essential part of the proof of Theorem
\ref{theorem-nerves-v2:5}.

\begin{proposition}
  \label{proposition-nerves-v2:2}
  Each lower locally constant mapping $\varphi:X\smap Y$ has a 
  $0$-skeletal selection.
\end{proposition}

\begin{proof}
  For each $x\in X$, take a point $y(x)\in \varphi(x)$, and set
  \begin{equation}
    \label{eq:nerves-v2:3}
    U(x)=\big\{z\in X: y(x)\in \varphi(z)\big\}.
  \end{equation}
  Then $\mathscr{U}=\{U(x):x\in X\}$ is an open cover of
  $X$. Moreover, for each $U\in \mathscr{U}$ there is a point $x_U\in
  X$ with $U=U(x_U)$. Since
  $|\mathscr{N}^0(\mathscr{U})|=\mathscr{U}$, we may define a map
  $u:|\mathscr{N}^0(\mathscr{U})|\to Y$ by $u(U)=y(x_U)$, for each
  $U\in \mathscr{U}$. If $x\in U\in \mathscr{U}$, then $x\in U(x_U)$
  and by (\ref{eq:nerves-v2:3}), we get that $u(U)=y(x_U)\in
  \varphi(x)$.
\end{proof}

\begin{proposition}
  \label{proposition-st-app-vgg-rev:1}
  Let $X$ be a paracompact space, and $\psi:X\smap Y$ be a mapping
  which has a $k$-skeletal selection, for some $k\geq 0$. Then $\psi$
  has a $k$-skeletal selection $u:|\mathscr{N}^k(\mathscr{U})|\to Y$
  for some open locally finite cover $\mathscr{U}$ of $X$. 
\end{proposition}

\begin{proof}
  Let $v:|\mathscr{N}^k(\mathscr{V})|\to Y$ be a $k$-skeletal
  selection for $\psi$, for some open cover $\mathscr{V}$ of
  $X$. Since $X$ is paracompact, the cover $\mathscr{V}$ has an open
  locally finite refinement $\mathscr{U}$. Let
  $r:\mathscr{N}(\mathscr{U})\to \mathscr{N}(\mathscr{V})$ be a refining
  map. Then by \eqref{eq:st-app-vgg-rev:9}, $u=v\circ |r|\uhr
  |\mathscr{N}^k(\mathscr{U})|: |\mathscr{N}^k(\mathscr{U})|\to Y$ is
  a $k$-skeletal selection for $\psi$ because
  $r(\Sigma_\mathscr{U}^k(x))\subset \Sigma_\mathscr{V}^k(x)$, for
  every $x\in X$. 
\end{proof}

A cover $\mathscr{V}$ of $X$ is a \emph{star-refinement} of a cover
$\mathscr{U}$ if the cover
$\mathscr{V}^*=\left\{V^*:V\in \mathscr{V}\right\}$ refines
$\mathscr{U}$, where
$V^*=\bigcup\{W\in \mathscr{V}:W\cap V\neq \emptyset\}$.  To reflect
this property, we shall say that a simplicial map
$\ell:\mathscr{N}(\mathscr{V})\to \mathscr{N}(\mathscr{U})$ is a
\emph{star-refining map} if $V^*\subset \ell(V)$, for each
$V\in \mathscr{V}$. Each star-refining map
$\ell:\mathscr{N}(\mathscr{V})\to \mathscr{N}(\mathscr{U})$ has the
property that
  \begin{equation}
    \label{eq:nerves-v2:2}
    \bigcup\sigma \subset  \bigcap \ell(\sigma),\quad
    \text{for each 
      $\sigma\in \mathscr{N}(\mathscr{V})$.}
  \end{equation}

\begin{proposition}
\label{proposition-nerves-v4:1}
Let $X$ be a paracompact space, $Y$ be a space, and
$\psi,\varphi:X\smap Y$ be such that $\varphi$ is lower locally
constant and $\psi \embed{k}\varphi$ for some $k\geq 0$. If $\psi$ has
a $k$-skeletal selection, then $\varphi$ has a $(k+1)$-skeletal
selection.
\end{proposition}

\begin{proof}
  By Proposition \ref{proposition-st-app-vgg-rev:1}, $\psi$ has a
  $k$-skeletal selection $u:|\mathscr{N}^k(\mathscr{U})|\to Y$ for
  some open locally finite cover $\mathscr{U}$ of $X$.  For each
  $\sigma\in \mathscr{N}(\mathscr{U})$, let
  $u_\sigma=u\uhr |\mathscr{N}^k(\sigma)|$ be the restriction of $u$
  over the subcomplex
  $|\mathscr{N}^k(\sigma)|=|\sigma|\cap |\mathscr{N}^k(\mathscr{U})|$.
  Then, whenever $\sigma\in \Sigma_\mathscr{U}^{k+1}(x)$ for some
  $x\in X$, the map $u_\sigma$ can be extended to a continuous map
  $u_{(x,\sigma)}:|\sigma|\to Y$ such that
  \begin{equation}
    \label{eq:nerves-v4:1}
    u_{(x,\sigma)}(|\sigma|)\subset \varphi(x). 
  \end{equation}
  Indeed, if $\sigma\in \Sigma_\mathscr{U}^k(x)$, then by
  \eqref{eq:st-app-vgg-rev:9},
  $u(|\sigma|)\subset \psi(x)\subset \varphi(x)$ and we can take
  $u_{(x,\sigma)}=u_\sigma$. If
  $\sigma\notin \Sigma_\mathscr{U}^k(x)$, then
  $| \mathscr{N}^k(\sigma)|=\bigcup\{|\tau|: \emptyset\neq
  \tau\subsetneq \sigma\}$ is homeomorphic to the $k$-sphere being the
  boundary of $|\sigma|$.  Hence, $u_\sigma$ has a continuous
  extension $u_{(x,\sigma)}:|\sigma|\to \varphi(x)$ because
  $u(| \mathscr{N}^k(\sigma)|)\subset u(
  |\Sigma_\mathscr{U}^k(x)|)\subset \psi(x)\embed{k} \varphi(x)$, see
  \eqref{eq:st-app-vgg-rev:9}.\medskip

  Now, whenever $x\in X$, set
  \begin{equation}
    \label{eq:nerves-v3:1}
    K(x)=\bigcup\left\{u_{(x,\sigma)}(|\sigma|):
      \sigma\in \Sigma_\mathscr{U}^{k+1}(x)\right\}.
  \end{equation}
  Then by (\ref{eq:nerves-v4:1}), $K(x)\subset \varphi(x)$; moreover,
  $K(x)$ is compact because $\mathscr{U}$ is locally finite and,
  therefore, $\Sigma_\mathscr{U}^{k+1}(x)$ contains finitely many
  simplices.  Since $\varphi$ is lower locally constant, for each
  $U\in \mathscr{U}$, each point $x\in U$ is contained in the open set
  \begin{equation}
    \label{eq:nerves-v3:3}
    W_{(x,U)}=\big\{z\in U: K(x)\subset \varphi(z)\big\}. 
  \end{equation}
  Since $X$ is paracompact, the cover
  $\{W_{(x,U)}: x\in U\in \mathscr{U}\}$ has an open star-refinement
  $\mathscr{V}$. So, there are maps $p:\mathscr{V}\to X$ and
  ${\ell:\mathscr{V}\to \mathscr{U}}$ such that
  \begin{equation}
    \label{eq:st-app-vgg-rev:4}
    V^*\subset W_{(p(V),\ell(V))},\quad \text{for every $V\in
      \mathscr{V}$.} 
  \end{equation}
  Accordingly, $\ell$ is a star-refining map because by
  \eqref{eq:nerves-v3:3},
  $V^*\subset W_{(p(V),\ell(V))}\subset \ell(V)$. Finally, take a map
  $q:\mathscr{N}(\mathscr{V})\to \mathscr{V}$ which selects from any
  simplex $\sigma\in \mathscr{N}(\mathscr{V})$ a vertex
  $q(\sigma)\in\sigma$, and next set
  $\pi=p\circ q:\mathscr{N}(\mathscr{V})\to X$. Then 
  \begin{equation}
    \label{eq:st-app-vgg-rev:2}
    \ell(\sigma)\in \Sigma_\mathscr{U}(\pi(\sigma)),\quad \sigma\in
    \mathscr{N}(\mathscr{V}), 
  \end{equation}
  because
  $\pi(\sigma)=p(q(\sigma))\in q(\sigma)\subset\bigcup \sigma\subset
  \bigcap \ell(\sigma)$, see (\ref{eq:nerves-v2:2}).\medskip

  We complete the proof as follows. Using \eqref{eq:nerves-v4:1} and
  \eqref{eq:st-app-vgg-rev:2}, one can define a continuous extension
  $v:|\mathscr{N}^{k+1}(\mathscr{V})|\to Y$ of the map
  $u\circ |\ell|\uhr
  |\mathscr{N}^k(\mathscr{V})|:|\mathscr{N}^k(\mathscr{V})|\to Y$ by
  $v\uhr |\sigma|=u_{(\pi(\sigma),\ell(\sigma))}$, for every
  $\sigma\in \mathscr{N}^{k+1}(\mathscr{V})$.  This $v$ is a
  $(k+1)$-skeletal selection for $\varphi$. Indeed, let
  $\sigma\in \Sigma_\mathscr{V}^{k+1}(x)$ for some $x\in X$. Then
  $x\in q(\sigma)$ because $q(\sigma)\in \sigma$, see
  \eqref{eq:st-app-vgg-rev:7}. Moreover, by
  \eqref{eq:st-app-vgg-rev:4},
  $q(\sigma)\subset [q(\sigma)]^*\subset
  W_{(\pi(\sigma),\ell(q(\sigma)))}$. Hence, by \eqref{eq:nerves-v3:1},
 \eqref{eq:nerves-v3:3} and \eqref{eq:st-app-vgg-rev:2},
  $ v(|\sigma|)=u_{(\pi(\sigma),\ell(\sigma))}(|\ell(\sigma)|)\subset
  K(\pi(\sigma))\subset \varphi(x)$.
\end{proof}


\begin{proof}[Proof of Theorem \ref{theorem-nerves-v2:5}]
  According to Propositions \ref{proposition-nerves-v2:2} and
  \ref{proposition-nerves-v4:1}, the mapping $\varphi_{n}$ has an
  $n$-skeletal selection $u:|\mathscr{N}^{n}(\mathscr{U})|\to Y$,
  for some open cover $\mathscr{U}$ of $X$.  Since $X$ is paracompact
  and $\dim(X)\le n$, the cover $\mathscr{U}$ has an open refinement
  $\mathscr{V}$ with
  $\mathscr{N}(\mathscr{V})=\mathscr{N}^{n}(\mathscr{V})$, see
  Remark \ref{remark-st-app-vgg-rev:4}. Let
  $r:\mathscr{N}(\mathscr{V})\to \mathscr{N}^{n}(\mathscr{U})$ be a
  refining map, and $g:X\to |\mathscr{N}(\mathscr{V})|$ be a canonical
  map for $\mathscr{V}$ which exists because $X$ is paracompact, see
  Theorems \ref{theorem-st-app-v3:1} and
  \ref{theorem-st-app-v12:1}. Then by Corollary
  \ref{corollary-st-app-vgg-rev:1}, the composite map
  $h=|r|\circ g:X\to |\mathscr{N}^{n}(\mathscr{U})|$ is a canonical
  map for $\mathscr{U}$. Finally, by \eqref{eq:st-app-vgg-rev:9} and
  Proposition \ref{proposition-st-app-v11:1}, the composite map
  $f=u\circ h:X\to Y$
\begin{center}
  \begin{tikzcd}
    &\lvert\mathscr{N}^{n}(\mathscr{U})\rvert \arrow[d, "u"]\\
    {X} \arrow[ur,  "h"] \arrow[r, rightsquigarrow, "\varphi"]
    & Y
  \end{tikzcd}
\end{center}
is a continuous selection for $\varphi$.
\end{proof}

\begin{remark}
  \label{remark-st-app-vgg-rev:4}
  Let $n\geq -1$ be an integer and $X$ be a normal space.  The
  \emph{order} of a cover $\mathscr{V}$ of $X$ doesn't exceed $n$ if
  $\bigcap\sigma=\emptyset$, for every $\sigma\subset \mathscr{V}$
  with $\card(\sigma)\geq n+2$; equivalently, if
  $\mathscr{N}(\mathscr{V})=\mathscr{N}^n(\mathscr{V})$. In these
  terms, the covering dimension of $X$ is at most $n$, written
  $\dim(X)\leq n$, if every finite open cover of $X$ has an open
  refinement $\mathscr{V}$ with
  $\mathscr{N}(\mathscr{V})=\mathscr{N}^n(\mathscr{V})$. According to
  a result of Dowker \cite[Theorem 3.5]{dowker:47}, $\dim(X)\leq n$ if
  and only if every locally finite open cover of $X$ has an open
  refinement $\mathscr{V}$ with
  $\mathscr{N}(\mathscr{V})=\mathscr{N}^n(\mathscr{V})$. In
  particular, for a paracompact space $X$, we have that
  $\dim(X)\leq n$ if and only if every open cover of $X$ has an open
  refinement $\mathscr{V}$ with
  $\mathscr{N}(\mathscr{V})=\mathscr{N}^n(\mathscr{V})$.
\end{remark}

\begin{remark}
  \label{remark-st-app-vgg-rev:5}
  A mapping $\psi:X\smap Y$ is a \emph{set-valued selection} (or
  \emph{set-selection}, or \emph{multi-selection}) for
  $\varphi:X\smap Y$ if $\psi(x)\subset \varphi(x)$, for all $x\in
  X$. In terms of set-valued selections, a mapping $\varphi:X\smap Y$
  has a $k$-skeletal selection, $k\geq 0$, if there exists an open
  cover $\mathscr{U}$ of $X$ and a continuous map
  $u:|\mathscr{N}^k(\mathscr{U})|\to Y$ such that the composite
  mapping $u\circ |\Sigma_\mathscr{U}^k|:X\smap Y$ 
\begin{center}
  \begin{tikzcd}
    &\lvert\mathscr{N}^{k}(\mathscr{U})\rvert \arrow[d,
    "u"]\\
    {X} \arrow[ur, rightsquigarrow, "\lvert\Sigma_\mathscr{U}^k\rvert"] \arrow[r,
    rightsquigarrow, "\varphi"] & Y
  \end{tikzcd}
\end{center}
is a set-valued selection for $\varphi:X\smap Y$.
\end{remark}

\section{Generating aspherical sequences of sets}
\label{sec:gener-asph-sequ}

For a point $y\in Y$ of a metric space $(Y,d)$ and $\varepsilon>0$,
let
\[
\mathbf{O}_\varepsilon(y)=\{z\in Y: d(z,y)<\varepsilon\}
\]
be the open $\varepsilon$-ball centred at $y$; and
$\mathbf{O}_\varepsilon(S)=\bigcup_{y\in S}\mathbf{O}_\varepsilon(y)$
be the $\varepsilon$-neighbourhood of a subset $S\subset Y$. Also,
recall that a map $f:X\to Y$ is an \emph{$\varepsilon$-selection} for
a mapping $\varphi:X\smap Y$ if $f(x)\in
\mathbf{O}_{\varepsilon}(\varphi(x))$ for every $x\in X$.\medskip

Throughout this section, $\delta:(0,+\infty)\to (0,+\infty)$ is a
fixed function. To this function, we associate the sequence of
iterated functions $\delta_n:(0,+\infty)\to (0,+\infty)$, $ n\geq 0$,
defined by
\begin{equation}
  \label{eq:st-app-v7:2}
  \delta_{0}(\varepsilon)=\varepsilon\quad\text{and}\quad
  \delta_{n+1}(\varepsilon)= \delta(\delta_{n}(\varepsilon)). 
\end{equation}

\begin{proposition}
  \label{proposition-st-app-v7:1}
  Let $(Y,d)$ be a metric space and
  ${S_0\subset S_1\subset\dots\subset S_{n}\subset Y}$
  be such that
  $\mathbf{O}_{\delta(\varepsilon)}(y)\cap S_k\embed{k}
  \mathbf{O}_\varepsilon(y)\cap S_{k+1}$, for every $y\in Y$ and
  $k< n$. Then
    \begin{equation}
      \label{eq:st-app-v7:3}
      \mathbf{O}_{\delta_{n-k}(\varepsilon)}(y)\cap S_k\embed{k}
      \mathbf{O}_{\delta_{n-k-1}(\varepsilon)}(y)\cap S_{k+1},\quad  k< n.
  \end{equation}
\end{proposition}

\begin{proof}
  Follows from the fact that
  $\delta_{n-k}(\varepsilon)=\delta\big(\delta_{n-k-1}(\varepsilon)\big)$,
  see (\ref{eq:st-app-v7:2}).
\end{proof}

We now have the following ``local'' version of Theorem
\ref{theorem-nerves-v2:5}.

\begin{theorem}
  \label{theorem-st-app-vgg-rev:2}
  Let $(Y,d)$ be a metric space, $X$ be a paracompact space with
  $\dim(X)\leq n$, and $\psi_k:X\smap Y$, $0\leq k\leq n$, be
  lower locally constant mappings such that
  $\mathbf{O}_{\delta(\varepsilon)}(y)\cap\psi_k(x)\embed{k}
  \mathbf{O}_\varepsilon(y)\cap \psi_{k+1}(x)$ for every $x\in X$,
  $y\in Y$ and $k< n$. Then for each continuous
  $\delta_{n}(\varepsilon)$-selection $g:X\to Y$ for $\psi_0$, there
  is a continuous selection $f:X\to Y$ for $\psi_{n}$ with
  $d(f(x),g(x))<\varepsilon$, for all $x\in X$.
\end{theorem}

\begin{proof}
  Let $g:X\to Y$ be a continuous $\delta_{n}(\varepsilon)$-selection
  for $\psi_0$. Next, for each $k\leq n$, define a set-valued mapping
  $\varphi_k$ by
  $\varphi_k(x)=\mathbf{O}_{\delta_{n-k}(\varepsilon)}(g(x))\cap
  \psi_k(x)$, $x\in X$. Since $g$ is a
  $\delta_{n}(\varepsilon)$-selection for $\psi_0$, the mapping
  $\varphi_0$ is nonempty-valued and, according to
  (\ref{eq:st-app-v7:3}), so is each $\varphi_k$, $k\leq n$. In fact,
  by (\ref{eq:st-app-v7:3}), the resulting sequence of mappings
  $\varphi_k:X\smap Y$, $0\leq k\leq n$, is aspherical.  Moreover,
  each $\varphi_k$ is lower locally constant because so are $\psi_k$
  and the mapping $x\to \mathbf{O}_{\delta_{n-k}(\varepsilon)}(g(x))$,
  $x\in X$ (see Proposition \ref{proposition-st-app-v15:2}).  Hence,
  by Theorem \ref{theorem-nerves-v2:5}, $\varphi_{n}$ has a continuous
  selection $f:X\to Y$ because $X$ is a paracompact space with
  $\dim(X)\leq n$. Evidently, $f$ is a selection for
  $\psi_{n}$ and, by (\ref{eq:st-app-v7:2}),
  $f(x)\in \mathbf{O}_{\delta_{n-n}(\varepsilon)}(g(x))=
  \mathbf{O}_{\delta_{0}(\varepsilon)}(g(x))=\mathbf{O}_\varepsilon(g(x))$,
  $x\in X$.
\end{proof}

We conclude this section with the following two applications of Theorem
\ref{theorem-st-app-vgg-rev:2} which will provide the main interface
between selections for l.s.c.\ mappings and Theorem
\ref{theorem-nerves-v2:5}, see Theorems \ref{theorem-st-app-v9:1} and
\ref{theorem-st-app-v17:1}.

\begin{corollary}
  \label{corollary-st-app-v7:2}
  Let $E$ be a normed space, and $\emptyset\neq S\subset T\subset E$
  be such that $S\embed{k} T$ and
  $\mathbf{O}_{\delta(\varepsilon)}(y)\cap S\embed{i}
  \mathbf{O}_\varepsilon(y)\cap S$, for every $y\in E$ and
  $0\leq i< k$. Then
  $\mathbf{O}_{\delta_k(\varepsilon)}(S)\embed{k}
  \mathbf{O}_\varepsilon(T)$.
\end{corollary}

\begin{proof}
  Let $\ell:\s^{k}\to \mathbf{O}_{\delta_k(\varepsilon)}(S)$ be a
  continuous map from the $k$-sphere $\s^k$. Consider the constant
  mappings $\psi_i(x)=S$, $x\in\s^k$ and $i\leq k$. Then $\ell$ is a
  continuous $\delta_k(\varepsilon)$-selection for $\psi_0$, and
  $\mathbf{O}_{\delta(\varepsilon)}(y)\cap \psi_i(x)\embed{i}
  \mathbf{O}_\varepsilon(y)\cap \psi_{i+1}(x)$ for every $x\in \s^k$,
  $y\in E$ and $i< k$. Hence, by Theorem
  \ref{theorem-st-app-vgg-rev:2}, there exists a continuous $q:\s^{k}\to
  S$ with $\|q(x)-\ell(x)\|<\varepsilon$, for every $x\in\s^k$. Let
  $h_{1}$ be the linear homotopy between $\ell$ and $q$, i.e.\
  $h_{1}(x,t)=t q(x)+ (1-t)\ell(x)$, whenever $(x,t)\in \s^{k}\times
  [0,1]$.  Then, $h_{1}(\s^{k}\times [0,1])\subset
  \mathbf{O}_{\varepsilon}(S)\subset
  \mathbf{O}_{\varepsilon}(T)$. Also, let $h_2:\s^k\times[0,1]\to T$
  be a homotopy between $q$ and a constant map, which exists because
  $S\embed{k}T$.  Finally, take $h$ to be the homotopy obtained by
  combining $h_{1}$ and $h_{2}$. Then $h$ is a homotopy of $\ell$ with
  a constant map over a subset of $\mathbf{O}_{\varepsilon}(T)$.
\end{proof}

\begin{corollary}
  \label{corollary-st-app-v8:1}
  Let $E$ be a normed space, and $\emptyset\neq S\subset T\subset E$
  be such that
  $\mathbf{O}_{\delta(\varepsilon)}(y)\cap T\embed{k}
  \mathbf{O}_\varepsilon(y)\cap T$ and
  $\mathbf{O}_{\delta(\varepsilon)}(y)\cap S\embed{i}
  \mathbf{O}_\varepsilon(y)\cap S$, for every $y\in E$ and
  $0\leq i< k$.  Define functions
  \begin{equation}
    \label{eq:st-app-v8:1}
    \eta(\varepsilon)=\delta(\varepsilon)/2\quad\text{and} \quad
    \lambda(\varepsilon,\mu)=\delta_k\big(\min\left\{\eta(\varepsilon),
      \mu\right\}\big),\
    \varepsilon,\mu>0.      
  \end{equation}
  Then $\mathbf{O}_{\eta(\varepsilon)}(y)\cap
  \mathbf{O}_{\lambda(\varepsilon,\mu)}(S)\embed{k}
  \mathbf{O}_\varepsilon(y)\cap \mathbf{O}_\mu(T)$, for every $y\in
  E$.
\end{corollary}

\begin{proof}
  Let
  $\ell:\s^{k}\to \mathbf{O}_{\eta(\varepsilon)}(y)\cap
  \mathbf{O}_{\lambda(\varepsilon,\mu)}(S)$ be a continuous map for
  some $y\in E$.  Then, precisely as in the previous proof, there
  exists a continuous map $q:\s^k\to S$ such that
  $\|q(x)-\ell(x)\|<\min\{\eta(\varepsilon),\mu\}$, for every
  $x\in \s^k$. Since $\eta(\varepsilon)=\delta(\varepsilon)/2$, see
  (\ref{eq:st-app-v8:1}), just like before, using a linear homotopy,
  we get that $\ell$ and $q$ are homotopic in
  $\mathbf{O}_{\delta(\varepsilon)}(y)\cap \mathbf{O}_\mu(S)$.
  Moreover $q$ is homotopic to a constant map in
  $\mathbf{O}_\varepsilon(y)\cap T$ because
  $q:\s^k\to \mathbf{O}_{\delta(\varepsilon)}(y)\cap S\subset
  \mathbf{O}_{\delta(\varepsilon)}(y)\cap T\embed{k}
  \mathbf{O}_\varepsilon(y)\cap T$.  Accordingly, $\ell$ is homotopic
  to a constant map in
  $\mathbf{O}_\varepsilon(y)\cap \mathbf{O}_\mu(T)$.
\end{proof}

\section{Selections for equi-$LC^{n}$-valued mappings}
\label{sec:select-lcn-valu}

In this section, to each $\Phi:X\smap Y$ we associate the mapping
$\overline{\Phi}:X\to \mathscr{F}(Y)$ defined by
$\overline{\Phi}(x)=\overline{\Phi(x)}$, $x\in X$. Moreover, for a
pair of mappings $\Phi,\Psi: X\smap Y$, we will use $\Phi\wedge\Psi$
to denote their intersection, i.e.\ the mapping which assigns to each
$x\in X$ the set $[\Phi\wedge\Psi](x)=\Phi(x)\cap \Psi(x)$. Finally,
to each $\varepsilon>0$ and a mapping $\Phi:X\smap Y$ in a metric
space $(Y,d)$, we will associate the mapping
$\mathbf{O}[\Phi,\varepsilon]:X\smap Y$ defined by
\begin{equation}
  \label{eq:st-app-v18:1}
 \mathbf{O}[\Phi,\varepsilon](x)=
 \mathbf{O}_\varepsilon(\Phi(x)),\quad \text{$x\in X$.}
\end{equation}
This convention will be also used in an obvious manner for usual maps
$f:X\to Y$ considering $f$ as the singleton-valued mapping
$x\to\{f(x)\}$, $x\in X$.  In these terms, for maps
$f,g:X\to Y$ and $\varepsilon,\mu>0$, we have that $f$ is a
$\mu$-selection for $\Phi:X\smap Y$ with $d(f(x),g(x))<\varepsilon$
for every $x\in X$, if and only if $f$ is a selection for the mapping
$\mathbf{O}[\Phi,\mu]\wedge\mathbf{O}[g,\varepsilon]$. \medskip

The following two constructions are due to Michael, see \cite[Lemma
11.3]{michael:56b} and \cite[Proof that Lemma 5.1 implies Theorem 4.1,
page 569]{michael:56b}. They reduce the selection problem for l.s.c.\
mappings to that of lower locally constant mappings. For completeness,
we sketch their proofs following the original arguments in
\cite{michael:56b}.

\begin{proposition}
  \label{proposition-st-app-v15:2}
  Let $(Y,d)$ be a metric space, $\Phi:X\smap Y$ be l.s.c.\ and
  $\varepsilon>0$. Then the mapping
  $\mathbf{O}[\Phi,\varepsilon]:X\smap Y$ is lower locally constant.
\end{proposition}

\begin{proof}
  Take $x_0\in X$ and a compact set
  $K\subset
  \mathbf{O}[\Phi,\varepsilon](x_0)=\mathbf{O}_\varepsilon(\Phi(x_0))$. Then
  $K\subset \mathbf{O}_\delta(S)$ for some finite subset
  $S\subset \Phi(x_0)$ and some $\delta>0$ with
  $\delta<\varepsilon$. Since $\Phi$ is l.s.c.,
  $U=\bigcap_{y\in S}\Phi^{-1}[\mathbf{O}_{\varepsilon-\delta}(y)]$ is
  an open set containing $x_0$. Moreover, $x\in U$ implies
  $S\subset \mathbf{O}_{\varepsilon-\delta}(\Phi(x))$ and, therefore,
  $K\subset \mathbf{O}_\delta(S)\subset
  \mathbf{O}_\varepsilon(\Phi(x))=\mathbf{O}[\Phi,\varepsilon](x)$.
\end{proof}

\begin{proposition}
  \label{proposition-st-app-v15:3}
  Let $(Y,d)$ be a complete metric space,
  ${\xi:(0,+\infty)\to (0,+\infty)}$ be a function with
  $\xi(\varepsilon)\leq \varepsilon$, and $\Phi:X\smap Y$ be a mapping
  such that for each continuous $\xi(\varepsilon)$-selection
  $g:X\to Y$ for $\Phi$ and $\mu>0$, then mapping
  ${\mathbf{O}[\Phi,\mu]\wedge \mathbf{O}[g,\varepsilon]}$ has a
  continuous selection. Then for every continuous
  $\xi(\varepsilon/2)$-selection $g:X\to Y$ for $\Phi$, the mapping
  $\overline{\Phi}\wedge\mathbf{O}[g,\varepsilon]$ also has a
  continuous selection.
\end{proposition}

\begin{proof}
  Let $f_1=g:X\to Y$ be a continuous
  $\xi\left(2^{-1}\varepsilon\right)$-selection for $\Phi$. By
  condition with $\mu=\xi\left(2^{-2}\varepsilon\right)$, the mapping
  $\mathbf{O}\left[\Phi,\xi\left(2^{-2}\varepsilon\right)\right]\wedge
  \mathbf{O}[f_1,2^{-1}\varepsilon]$ has a continuous selection
  $f_2:X\to Y$. Thus, by induction, there exists a sequence of
  continuous maps $f_n:X\to E$ such that $f_{n+1}$ is a selection for
  $\mathbf{O}\left[\Phi,\xi\left(2^{-(n+1)}\varepsilon\right)\right]\wedge
  \mathbf{O}\left[f_n,2^{-n}\varepsilon\right]$, for every $n\in\N$.
  Then the sequence $\{f_n:n\in \N\}$ is uniformly Cauchy because
  $d(f_{n+1}(x),f_n(x))<2^{-n}\varepsilon$, $x\in X$. Hence, it
  converges uniformly to some continuous map $f:X\to Y$ because
  $(Y,d)$ is complete. Since
  $\xi\left(2^{-n}\varepsilon\right)\leq 2^{-n}\varepsilon$, each
  $f_{n}$ is a $2^{-n}\varepsilon$-selection for $\Phi$ being a
  selection for
  $\mathbf{O}\left[\Phi,\xi\left(2^{-n}\varepsilon\right)\right]$, see
  (\ref{eq:st-app-v18:1}). Hence, $d(f(x),\Phi(x))=0$, for each
  $x\in X$. Finally, we also have that
  \[
  d(f(x),g(x))\leq\sum_{n=1}^\infty
  d(f_{n+1}(x),f_{n}(x))<\sum_{n=1}^\infty2^{-n}\varepsilon=\varepsilon,\quad
  x\in X.  \qedhere
  \]
\end{proof}

Let $n\geq -1$. A family $\mathscr{S}$ of subsets of a metric space
$(Y,d)$ is called \emph{uniformly equi-$LC^{n}$} \cite{michael:56b} if
for every $\varepsilon>0$ there exists $\delta(\varepsilon)>0$ such
that, for every $S\in\mathscr{S}$, every continuous map of the
$k$-sphere ($k\leq n$) in $S$ of diameter$\ < \delta(\varepsilon)$ can
be extended to continuous map of the $(k+1)$-ball into a subset of $S$
of diameter$\ <\varepsilon$.  Just as in the case of equi-$LC^n$
families, a family $\mathscr{S}$ is uniformly equi-$LC^{-1}$ iff it
consists of nonempty sets.  For such a family $\mathscr{S}$, by
replacing $\delta(\varepsilon)$ with $\frac{\delta(\varepsilon)}2$, we
get that $\mathscr{S}$ is uniformly equi-$LC^{n}$ if there exists a
function $\delta:(0,+\infty)\to (0,+\infty)$ such that
\begin{equation}
  \label{eq:st-app-v7:4}
  \mathbf{O}_{\delta(\varepsilon)}(y)\cap S\embed{k} \mathbf{O}_\varepsilon(y)\cap
  S,\quad \text{for every $S\in \mathscr{S}$, $y\in Y$ and $0\leq k\leq n$.}
\end{equation}
Evidently, we may further assume that
$\delta(\varepsilon)\leq \varepsilon$, for every $\varepsilon>0$.
Based on this and the results of the previous section, we now have the
following two applications of Theorem \ref{theorem-nerves-v2:5}. The
first one gives a simplified proof of \cite[Theorem 4.1]{michael:56b}.

\begin{theorem}
  \label{theorem-st-app-v9:1}
  Let $E$ be a Banach space and $\mathscr{S}$ be a uniformly
  equi-$LC^{n}$ family of subsets of $E$. Then there exists
  a function $\gamma:(0,+\infty)\to (0,+\infty)$ with the following
  property\textup{:} If $X$ is a paracompact space with $\dim(X)\le
  n+1$, $\Phi:X\to \mathscr{S}$ is l.s.c.\ and $g:X\to E$ is a
  continuous $\gamma(\varepsilon)$-selection for $\Phi$, then
  $\overline{\Phi}\wedge \mathbf{O}[g,\varepsilon]$ has a continuous
  selection.
\end{theorem}

\begin{proof}
  Let $\delta(\varepsilon)\leq \varepsilon$ be as in
  (\ref{eq:st-app-v7:4}) with respect to the family
  $\mathscr{S}$. Also, let $\lambda(\varepsilon,\mu)$ and
  $\eta(\varepsilon)$ be as in (\ref{eq:st-app-v8:1}) applied to this
  particular function $\delta(\varepsilon)$. Next, define functions
  $\eta_k(\varepsilon)$ and $\lambda_k(\varepsilon,\mu)$,
  $0\leq k\leq n+1$, by
  \begin{equation}
    \label{eq:st-app-v9:1}
    \begin{cases}
      \eta_{n+1}(\varepsilon)=\varepsilon &\text{and}\quad
      \eta_{k}(\varepsilon)=\eta\big(\eta_{k+1}(\varepsilon)\big)\\
      \lambda_{n+1}(\varepsilon,\mu)=\mu &\text{and}\quad
      \lambda_{k}(\varepsilon,\mu)=
      \lambda\big(\eta_{k+1}(\varepsilon),\lambda_{k+1}(\varepsilon,\mu)\big).
    \end{cases}
  \end{equation}
  Then $\gamma(\varepsilon)=\eta_{0}(\varepsilon/2)$ is as
  required. Indeed, let $X$ and $\Phi$ be as in the theorem.  Applying
  Proposition \ref{proposition-st-app-v15:3} with
  $\xi(\varepsilon)=\eta_{0}(\varepsilon)$, it will be now
  sufficient to show that for every $\mu>0$ and a continuous
  $\eta_{0}(\varepsilon)$-selection $g:X\to E$ for $\Phi$, the
  mapping $\mathbf{O}[\Phi,\mu]\wedge \mathbf{O}[g,\varepsilon]$ has a
  continuous selection. To this end, for every $0\le k\le n+1$, let
  $\varphi_k=\mathbf{O}[\Phi,\lambda_{k}(\varepsilon,\mu)]\wedge
  \mathbf{O}[g,\eta_{k}(\varepsilon)]$. According to Proposition
  \ref{proposition-st-app-v15:2}, each $\varphi_{k}$ is lower locally
  constant. Moreover, the resulting sequence of mappings
  $\varphi_k:X\smap E$, $0\leq k\leq n+1$, is aspherical because by
  (\ref{eq:st-app-v7:4}) and Corollary \ref{corollary-st-app-v8:1},
  \begin{align*}
    \varphi_k(x) =&\ 
                    \mathbf{O}_{\lambda_{k}(\varepsilon,\mu)}(\Phi(x))\cap
                    \mathbf{O}_{\eta_{k}(\varepsilon)}(g(x))\\ 
    =&\
       \mathbf{O}_{\lambda(\eta_{k+1}(\varepsilon),\lambda_{k+1}(\varepsilon,\mu))}(\Phi(x))
    \cap \mathbf{O}_{\eta(\eta_{k+1}(\varepsilon))}(g(x))\\
                  &\embed{k} 
                    \mathbf{O}_{\lambda_{k+1}(\varepsilon,\mu)}(\Phi(x))\cap
                    \mathbf{O}_{\eta_{k+1}(\varepsilon)}(g(x)) 
                    =\varphi_{k+1}(x),\quad k\leq n.
  \end{align*}
  Hence by Theorem \ref{theorem-nerves-v2:5},
  \[
    \varphi_{n+1}=\mathbf{O}[\Phi,\lambda_{n+1}(\varepsilon,\mu)]\wedge
  \mathbf{O}[g,\eta_{n+1}(\varepsilon)]=\mathbf{O}[\Phi,\mu]\wedge
  \mathbf{O}[g,\varepsilon]
\]
has a continuous selection. The proof is complete.
\end{proof}

\begin{theorem}
\label{theorem-st-app-v17:1}
Let $X$ be a paracompact space with $\dim(X)\le n+1$, $E$ be a Banach
space, and $\Phi_{k}:X\smap E$, $0\leq k\leq n+1$, be a sequence of
l.s.c.\ mappings such that $\{\Phi_{k}(x): x\in X\}$ is uniformly
equi-$LC^{k}$ and $\Phi_{k} \embed{k}\Phi_{k+1}$ for every $k\leq
n$. Then $\Phi_{n+1}$ has a continuous $\varepsilon$-selection, for
every $\varepsilon>0$.
\end{theorem}

\begin{proof}
  According to (\ref{eq:st-app-v7:4}) and Corollary
  \ref{corollary-st-app-v7:2}, for each $0\leq k\leq n$ there exists a
  function $\delta_k:(0,+\infty)\to (0,+\infty)$ such that
  \begin{equation}
    \label{eq:st-app-v17:1}
    \mathbf{O}_{\delta_k(\varepsilon)}(\Phi_k(x))\embed{k}
    \mathbf{O}_\varepsilon(\Phi_{k+1}(x)),\quad x\in X.
  \end{equation}
  Next, define functions $\gamma_k:(0,+\infty)\to (0,+\infty)$,
  $0\leq k\leq n+1$, by
  \begin{equation}
    \label{eq:st-app-v17:2}
    \gamma_{n+1}(\varepsilon)=\varepsilon\quad \text{and}\quad \gamma_{k}(\varepsilon)=
    \delta_{k}(\gamma_{k+1}(\varepsilon)),\quad k\leq n.
  \end{equation}
  Finally, define a sequence of mappings $\varphi_k:X\smap E$ by
  $\varphi_k=\mathbf{O}[\Phi_k,\gamma_{k}(\varepsilon)]$. It now
  follows from (\ref{eq:st-app-v17:1}) and (\ref{eq:st-app-v17:2})
  that
  \begin{align*}
    \varphi_k(x)=\mathbf{O}_{\gamma_{k}(\varepsilon)}(\Phi_k(x))
    &
      = \mathbf{O}_{\delta_k(\gamma_{k+1}(\varepsilon))}(\Phi_k(x))\\
    &\embed{k}
      \mathbf{O}_{\gamma_{k+1}(\varepsilon)}(\Phi_{k+1}(x))=\varphi_{k+1}(x),\quad
      k\leq n. 
  \end{align*}
  Hence, the mappings $\varphi_k$, $0\leq k\leq n+1$, form an
  aspherical sequence. Moreover, by Proposition
  \ref{proposition-st-app-v15:2}, each $\varphi_k$ is lower locally
  constant. Since $\dim(X)\leq n+1$, by Theorem
  \ref{theorem-nerves-v2:5},
  $\varphi_{n+1}=\mathbf{O}[\Phi_{n+1},
  \gamma_{n+1}(\varepsilon)]=\mathbf{O}[\Phi_{n+1},\varepsilon]$ has a
  continuous selection, i.e. $\Phi_{n+1}$ has a continuous
  $\varepsilon$-selection.
\end{proof}

We are also ready for the proof of Theorem \ref{theorem-st-app-v10:1}.

\begin{proof}[Proof of Theorem \ref{theorem-st-app-v10:1}]
  Let $X$, $Y$, $\mathscr{S}\subset \mathscr{F}(Y)$ and $\Phi:X\to
  \mathscr{S}$ be as in that theorem. Since $\mathscr{S}$ is
  equi-$LC^n$, by \cite[Theorem
  1]{dugundji-michael:56} (see, also, \cite[Proposition
  2.1]{michael:56b}), $\bigcup\mathscr{S}$ can be embedded into a
  Banach space $E$  such that $\mathscr{S}\subset \mathscr{F}(E)$ is
  uniformly equi-$LC^n$.  Then by Theorem \ref{theorem-st-app-v17:1},
  applied with $\Phi_k=\Phi$, $0\leq k\leq n+1$, the mapping $\Phi$
  has a continuous $\varepsilon$-selection, for every
  $\varepsilon>0$. Hence, by Theorem~\ref{theorem-st-app-v9:1},
  $\overline{\Phi}=\Phi$ has a continuous selection as well. 
\end{proof}

Another application of Theorems \ref{theorem-st-app-v9:1} and
\ref{theorem-st-app-v17:1}  is the following generalisation of
Theorem \ref{theorem-st-app-v10:1}, see \cite{schepin-brodsky:96},
\cite[Theorem 7.2]{repovs-semenov:98} and \cite[Corollary
7.10]{gutev:05}.

\begin{corollary}
  \label{corollary-st-app-v18:1}
  Let $X$ be a paracompact space with $\dim(X)\le n+1$, $Y$ be a
  completely metrizable space, and $\Phi_{k}:X\to \mathscr{F}(Y)$,
  $0\leq k\leq n+1$, be a sequence of l.s.c.\ mappings such that
  $\Phi_{k} \embed{k}\Phi_{k+1}$ for $k\leq n$, while each family
  ${\{\Phi_{k}(x): x\in X\}}$, for $k\leq n+1$, is
  equi-$LC^{k}$. Then $\Phi_{n+1}$ has a continuous selection.
\end{corollary}

\begin{proof}
  As before, the proof is reduced to the case when $Y=E$ is a Banach
  space, and each family
  $\{\Phi_{k}(x):x\in X\}\subset \mathscr{F}(E)$, $0\le k\le n+1$, is
  uniformly equi-$LC^{k}$. Then by Theorem \ref{theorem-st-app-v17:1},
  $\Phi_{n+1}$ has a continuous $\varepsilon$-selection, for every
  $\varepsilon>0$.  Finally, by Theorem \ref{theorem-st-app-v9:1},
  $\Phi_{n+1}$ also has a continuous selection.
\end{proof}


\end{document}